\documentclass[12pt,a4paper]{article}
\usepackage[utf8x]{inputenc}
\usepackage{amsmath}
\usepackage{amsfonts}
\usepackage{amssymb}
\usepackage{amsthm}
\usepackage{mathrsfs}
\usepackage{fancyhdr}
\usepackage{graphicx}
\usepackage[usenames,x11names]{xcolor}
\usepackage{arabtex}
\usepackage{framed}
\usepackage[top=2cm,bottom=2cm,margin=2cm]{geometry}
\usepackage{cancel}
\usepackage{array}   

\usepackage[dvipdfm,pdfstartview=FitH,colorlinks=true,linkcolor=Red4,urlcolor=Red4,%
filecolor=green,citecolor=red]{hyperref}

\title{On the sum of the values of a polynomial at natural numbers which form a decreasing arithmetic progression}
\author{\sc Bakir FARHI \\
Laboratoire de Mathématiques appliquées \\
Faculté des Sciences Exactes \\
Université de Bejaia, 06000 Bejaia, Algeria \\[1mm]
\href{mailto:bakir.farhi@gmail.com}{bakir.farhi@gmail.com} \\[1mm]
\url{http://farhi.bakir.free.fr/}
}
\date{}

\let\up=\textsuperscript

\def\R{{\mathbb R}}

\def\N{{\mathbb N}}
\def\Z{{\mathbb Z}}

\def\E{\mathscr{E}}

\def\int{\mathrm{Int}}

\def\restmod#1#2{#1\ (\mathrm{mod}\ #2)} 

\newcommand{\codim}[1]{\mathrm{codim}\,#1}

\def\EMdash{\leavevmode\hbox to 10.6mm{\vrule height .63ex depth -.59ex
    width 10mm\hfill}}

\theoremstyle{plain}
\numberwithin{equation}{section}
\newtheorem{thm}{Theorem}[section]

\newtheorem{prop}[thm]{Proposition}
\newtheorem{coll}[thm]{Corollary}

\newtheorem{thmn}{Theorem}      

\theoremstyle{definition}
\newtheorem{defi}[thm]{Definition}

\theoremstyle{remark}
\newtheorem{rmk}[thm]{Remark}
\newtheorem{rmks}[thm]{Remarks}

\begin{document}
\maketitle

\vspace*{-11cm}

\begin{flushleft}
\emph{A tribute to Ibn al-Banna al-Marrakushi \\
on the 700\up{th} anniversary of his death}
\end{flushleft}

\vspace*{9cm}

\begin{abstract}
The purpose of this paper consists to study the sums of the type \linebreak $P(n) + P(n - d) + P(n - 2 d) + \dots$, where $P$ is a real polynomial, $d$ is a positive integer and the sum stops at the value of $P$ at the smallest natural number of the form $(n - k d)$ ($k \in \N$). Precisely, for a given $d$, we characterize the $\R$-vector space $\E_d$ constituting of the real polynomials $P$ for which the above sum is polynomial in $n$. The case $d = 2$ is studied in more details. In the last part of the paper, we approach the problem through formal power series; this inspires us to generalize the spaces $\E_d$ and the underlying results. Also, it should be pointed out that the paper is motivated by the curious formula: $n^2 + (n - 2)^2 + (n - 4)^2 + \dots = \frac{n (n + 1) (n + 2)}{6}$, due to Ibn al-Banna al-Marrakushi (around 1290).      
\end{abstract}

\section{Introduction and Notation}\label{sec1}

Throughout this paper, we let $\N^*$ denote the set $\N \setminus \{0\}$ of positive integers. For $x \in \R$, we let $\lfloor x\rfloor$ denote the integer part of $x$. For a given prime number $p$, we let $\vartheta_p$ denote the usual $p$-adic valuation. For a given positive integer $n$, we call \textit{the odd part} of $n$, which we denote by $\mathrm{odd}(n)$, the greatest odd positive divisor of $n$; that is $\mathrm{odd}(n) = n / 2^{\vartheta_2(n)}$. For a given rational number $r$, \textit{the denominator} of $r$ designates the smallest positive integer $d$ such that $d r \in \Z$. Next, we let $\R[X]$ denote the $\R$-vector space of polynomials in $X$ with coefficients in $\R$. For $P \in \R[X]$, we let $\deg P$ denote the degree of $P$. Besides, for $n \in \N$, we let $\R_n[X]$ denote the $\R$-linear subspace of $\R[X]$ constituting of real polynomials with degree $\leq n$, and we set by convention $\R_{-1}[X] := \{0_{\R[X]}\}$ (the null vector subspace of $\R[X]$). For $n \in \N$, we let also $\binom{X}{n}$ denote the polynomial of $\R_n[X]$ defined by $\binom{X}{n} := \frac{X (X - 1) \cdots (X - n + 1)}{n!}$. The dimension of a vector space $E$ is denoted by $\dim E$ and the codimension of a linear subspace $F$ of a given vector space $E$ is denoted by $\codim_E F$.

Further, \textit{the Bernoulli polynomials} $B_n(X)$ ($n \in \N$) can be defined by their exponential generating function:
$$
\frac{t e^{X t}}{e^t - 1} = \sum_{n = 0}^{+ \infty} B_n(X) \frac{t^n}{n!}
$$
and \textit{the Bernoulli numbers} $B_n$ are the values of the Bernoulli polynomials at $X = 0$; that is $B_n := B_n(0)$ ($\forall n \in \N$). To make the difference between the Bernoulli polynomials and the Bernoulli numbers, we always put the indeterminate $X$ in evidence when it comes to polynomials. Similarly, \textit{the Euler polynomials} $E_n(X)$ ($n \in \N$) can be defined by their exponential generating function:
$$
\frac{2 e^{X t}}{e^t + 1} = \sum_{n = 0}^{+ \infty} E_n(X) \frac{t^n}{n!} .
$$
\textit{The Genocchi numbers} $G_n$ ($n \in \N$) can also be defined by their exponential generating function:
$$
\frac{2 t}{e^t + 1} = \sum_{n = 0}^{+ \infty} G_n \frac{t^n}{n!} .
$$
The famous Genocchi theorem \cite{gen} states that the $G_n$'s are all integers. The Bernoulli polynomials and numbers, the Euler polynomials and the Genocchi numbers have been studied by several authors (see e.g., \cite{com,far1,far2,nie}) and have many important and remarkable properties; among them we just cite the followings:
\begin{align}
B_n(X + 1) - B_n(X) = n X^{n - 1} ~~~~~~~~~~ (\forall n \in \N) , \label{eqJ} \\[1mm]
B_n\left(\frac{1}{2}\right) = \left(\frac{1}{2^{n - 1}} - 1\right) B_n ~~~~~~~~~~ (\forall n \in \N) , \label{eqJJ} \\[1mm]
G_n = 2 \left(1 - 2^n\right) B_n = n E_{n - 1}(0) ~~~~~~~~~~ (\forall n \in \N^*) . \label{eqJJJ}
\end{align}

In Arabic mathematics, there is a long tradition of determining closed forms for the sums of the type: $(1^k + 2^k + \dots + n^k)$ ($k \in \N^*$). Apart from the sum $(1 + 2 + \dots + n)$ for which the closed form has been known since antiquity (equal to $\frac{n (n + 1)}{2}$), we can cite
\begin{eqnarray*}
1^2 + 2^2 + \dots + n^2 & = & \frac{n (n + 1) (2 n + 1)}{6} ~~~~~~~~~~~~~~~~~~~~~~~~~~~~~~ \text{(al-Karaji)} , \\[2mm]
1^3 + 2^3 + \dots + n^3 & = & \frac{n^2 (n + 1)^2}{4} ~~~~~~~~~~~~~~~~~~~~~~~~~~~~~~~~~~~~~~ \text{(al-Karaji)} , \\[2mm]
1^4 + 2^4 + \dots + n^4 & = & \frac{n (n + 1) (2 n + 1) (3 n^2 + 3 n - 1)}{30} ~~~~~~~~~~ \text{(Ibn al-Haytham)} .
\end{eqnarray*}
In the course of the seventeenth century, it became increasingly evident that for any nonzero polynomial $P \in \R[X]$, the sum $P(0) + P(1) + \dots + P(n)$ ($n \in \N$) is polynomial in $n$ with degree $(\deg P + 1)$. Equivalently, for any $P \in \R[X]$, there exists $Q \in \R[X]$ satisfying $Q(X + 1) - Q(X) = P(X)$. More explicitly, Jacob Bernoulli \cite{ber} obtained the following remarkable formula:
$$
0^k + 1^k + \dots + n^k = \frac{1}{k + 1} \sum_{i = 0}^{k} \binom{k + 1}{i} B_i n^{k + 1 - i} + n^k ~~~~~~~~~~ (\forall k , n \in \N) . 
$$
For a given positive integer $d$ and a given polynomial $P \in \R[X]$, let us write:
\begin{equation}\label{eqs}
P(n) + P(n - d) + P(n - 2 d) + \dots ~~~~~~~~~~ (n \in \N) \tag{$\Sigma$}
\end{equation}
to designate the sum of the values of $P$ from $i = n$ to the smallest natural number of the form $(n - k d)$ ($k \in \N$), by decreasing $i$ in each step by $d$; that is the sum $\sum_{0 \leq k \leq \frac{n}{d}} P(n - k d)$. Although the sums \eqref{eqs} are polynomial in $n$ when $d = 1$ (as seen above), this is not always the case for $d \geq 2$. Indeed, we have for example (for $n \in \N$):
$$
n + (n - 2) + (n - 4) + \dots = \left\lfloor\frac{n + 1}{2}\right\rfloor \left(\left\lfloor\frac{n}{2}\right\rfloor + 1\right) ,
$$
which is not polynomial in $n$. In his pioneer book entitled ``Raf' al-\d{h}ijab 'an wujuh a'mal al-\d{h}isab'' (see e.g., \cite{aba}), meaning ``Lifting the veil from the faces of the workings of Arithmetic'', the great Arabic mathematician Ibn al-Banna al-Marrakushi (1256 - 1321) found a case when the sum \eqref{eqs} becomes polynomial in $n$ with $d \neq 1$. He precisely obtained the following curious formula:
\begin{equation}\label{eqbanna}
n^2 + (n - 2)^2 + (n - 4)^2 + \dots = \frac{n (n + 1) (n + 2)}{6} ~~~~~~~~ (\forall n \in \N) .
\end{equation}
More interestingly is the proof of Ibn al-Banna of his formula \eqref{eqbanna}. A natural way to prove \eqref{eqbanna} consists to distinguish two cases according to the parity of $n$. But Ibn al-Banna did not do this; he proved \eqref{eqbanna} in one step! To do so, he remarked that any perfect square number can be written as a sum of two consecutive triangular numbers; precisely, we have for all $n \in \N$:
$$
n^2 = T_n + T_{n - 1} ,
$$
where $T_k := \frac{k (k + 1)}{2}$ is the triangular number of order $k$. By admitting negative numbers (one thing that Ibn al-Banna himself avoids!\footnote{Although his predecessor al-Samawal (1130-1180) uses negative numbers with great skill.}), the Ibn al-Banna method for proving \eqref{eqbanna} becomes as follows:
\begin{framed}
\textit{\fontsize{10}{10}{%
\noindent For any $n \in \N$, we have
\begin{align*}
n^2 + (n - 2)^2 + (n - 4)^2 + \dots & = (T_n + T_{n - 1}) + (T_{n - 2} + T_{n - 3}) + (T_{n - 4} + T_{n - 5}) + \dots \\
& = T_n + T_{n - 1} + T_{n - 2} + \dots ,
\end{align*}
where the last sum stops at $T_{-1}$ when $n$ is even and stops at $T_0$ when $n$ is odd. But because $\boxed{T_{-1} = 0}$, we can stop the sum in question at $T_0$ in both cases. So we have
\begin{align*}
n^2 + (n - 2)^2 + (n - 4)^2 + \dots & = T_0 + T_1 + T_2 + \dots + T_n \\
& = \sum_{k = 0}^{n} \frac{k (k + 1)}{2} \\
& = \frac{n (n + 1) (n + 2)}{6} ,
\end{align*}
as required.
}}
\end{framed}
\noindent Looking closely at the previous Ibn al-Banna proof of Formula \eqref{eqbanna}, we see that the fundamental property that makes it work is $T_{-1} = 0$. So, by proceeding in the same way, we can establish other formulas of the same type as \eqref{eqbanna} and as curious as it is, with other values of $d$. I have obtained for example the following:
\begin{equation}\label{eqbakir}
n \left(n^2 + 1\right) + (n - 3) \left((n - 3)^2 + 1\right) + (n - 6) \left((n - 6)^2 + 1\right) + \dots = \frac{n (n + 1) (n + 2) (n + 3)}{12} .
\end{equation}
To prove \eqref{eqbakir}, let us observe that for any $n \in \N$, we have
$$
n (n^2 + 1) = \alpha_n + \alpha_{n - 1} + \alpha_{n - 2} ,
$$
where $\alpha_k := \frac{k (k + 1) (k + 2)}{3}$ ($\forall k \in \Z$). Thus we have
\begin{align*}
n \left(n^2 + 1\right) + (n - 3) \left((n - 3)^2 + 1\right) + (n - 6) \left((n - 6)^2 + 1\right) + \dots \\
& \hspace*{-5cm} = \left(\alpha_n + \alpha_{n - 1} + \alpha_{n - 2}\right) + \left(\alpha_{n - 3} + \alpha_{n - 4} + \alpha_{n - 5}\right) + \dots \\
& \hspace*{-5cm} = \alpha_n + \alpha_{n - 1} + \alpha_{n - 2} + \dots ,
\end{align*}
where the last sum stops at $\alpha_{-2}$ when $n \equiv \restmod{0}{3}$, stops at $\alpha_{-1}$ when $n \equiv \restmod{1}{3}$ and stops at $\alpha_0$ when $n \equiv \restmod{2}{3}$. But because $\boxed{\alpha_{-2} = \alpha_{-1} = 0}$, we can stop the sum in question at $\alpha_0$ in both cases. So we get
\begin{align*}
n \left(n^2 + 1\right) + (n - 3) \left((n - 3)^2 + 1\right) + (n - 6) \left((n - 6)^2 + 1\right) + \dots & = \alpha_{0} + \alpha_{1} + \dots + \alpha_n \\
& = \sum_{k = 0}^{n} \frac{k (k + 1) (k + 2)}{3} \\
& = \frac{n (n + 1) (n + 2) (n + 3)}{12} , 
\end{align*}
as required. Actually, we can even find a sum of type \eqref{eqs} with $d = 3$ which is polynomial in $n$ with degree only $2$. To do so, it suffices to take in the previous reasoning $\alpha_k = (k + 1) (k + 2)$ (instead of $\alpha_k = \frac{1}{3} k (k + 1) (k + 2)$). We then obtain that for $P_0(n) = 3 n^2 + 3 n + 2$ ($\forall n \in \N$), we have
$$
P_0(n) + P_0(n - 3) + P_0(n - 6) + \dots = \frac{(n + 1) (n + 2) (n + 3)}{3} ~~~~~~~~~~ (\forall n \in \N) .
$$
However, I personally find Formula \eqref{eqbakir} more elegant than the last! Actually, the mystery of Formula \eqref{eqbanna} can be explained otherwise by using the following immediate identity which holds for any $n \in \N$:
$$
n^2 + (n - 2)^2 + (n - 4)^2 + \dots = \frac{1}{2} \sum_{k = 0}^{n} (n - 2 k)^2 ;
$$
but unfortunately, this simple technique cannot be used to explain other identities of the same type as \eqref{eqbanna} (it is incapable for example to explain \eqref{eqbakir}).

Throughout this paper, for a given positive integer $d$, we let $\E_d$ denote the set of all polynomials $P \in \R[X]$ for which the sum \eqref{eqs} is polynomial in $n$. For example, we have (according to what is said above): $X \not\in \E_2 , X^2 \in \E_2 , X (X^2 + 1) \in \E_3$ and $3 X^2 + 3 X + 2 \in \E_3$. It is immediate that $\E_d$ ($d \in \N^*$) is a $\R$-linear subspace of $\R[X]$. This paper is devoted to studying the detailed structure of $\E_d$; in particular to determine for it an explicit basis and to specify its codimension in $\R[X]$. The case $d = 1$ is trivial because we know that the sum \eqref{eqs} is always polynomial in $n$ when $d = 1$; thus $\E_1 = \R[X]$. The case $d = 2$ is studied in more detail by providing for $\E_2$ a basis which is close (in a some sense) to the canonical basis of $\R[X]$; this will connect us with the sequence of the values of the Euler polynomials at the origin. Then, we go on to prove the generalizability of the Ibn al-Banna method (exposed above). Precisely, we will show that whenever the sum \eqref{eqs} is polynomial in $n$, the Ibn al-Banna method can be used to determine its closed form. We conclude the paper by giving (briefly) another approach (using convolution products and power series) to studying the spaces $\E_d$. This approach inspires us moreover a generalization of the spaces $\E_d$, and some extensive results are given without proofs. It should also be noted that during the description of the second approach of studying and generalizing the spaces $\E_d$, some other notations will be introduced. Especially, for more clarity and convenience, the sum \eqref{eqs} is denoted by $S_{P , d}(n)$, highlighting each of the parameters $P , d$ and $n$. 

\section{The results and the proofs}
Our main result is the following:
\begin{thm}\label{t1}
Let $d$ be a positive integer. Then, a polynomial $P \in \R[X]$ belongs to $\E_d$ if and only if it has the form:
\begin{equation}\label{eqI}
P(X) = \binom{X + d}{d} f(X + d) - \binom{X}{d} f(X) ,
\end{equation}
where $f \in \R[X]$. Besides, if \eqref{eqI} holds then we have for any natural number $n$:
$$
P(n - d) + P(n - 2 d) + P(n - 3 d) + \dots = \binom{n}{d} f(n) .
$$
\end{thm}
\begin{proof}
Let $P \in \R[X]$. \\[1mm]
\textbullet{} Suppose that $P \in \E_d$; that is there exists $S \in \R[X]$ for which we have for any natural number $n$:
\begin{equation}\label{eq1}
P(n) + P(n - d) + P(n - 2 d) + \dots = S(n) .
\end{equation}
Using \eqref{eq1}, we have in particular:
$$
S(0) = P(0) ~,~ S(1) = P(1) ~,~ \dots ~,~ S(d - 1) = P(d - 1) .
$$
This shows that the real polynomial $(S(X) - P(X))$ vanishes at $0 , 1 , \dots , d - 1$. Consequently, there exists $f_0 \in \R[X]$ such that:
$$
S(X) - P(X) = X (X - 1) \cdots (X - d + 1) f_0(X) = \binom{X}{d} f(X) ,
$$
where $f := d! f_0 \in \R[X]$. By specializing the obtained polynomial equality $S(X) - P(X) = \binom{X}{d} f(X)$ to $X = n \in \N$, we get (according to \eqref{eq1}):
\begin{equation}\label{eq2}
P(n - d) + P(n - 2 d) + P(n - 3 d) + \dots = \binom{n}{d} f(n) ~~~~~~~~~~ (\forall n \in \N) .
\end{equation}
Next, by applying \eqref{eq2} for $(n + d)$ (instead of $n \in \N$), we get
\begin{equation}\label{eq3}
P(n) + P(n - d) + P(n - 2 d) + \dots = \binom{n + d}{d} f(n + d) ~~~~~~~~~~ (\forall n \in \N) .
\end{equation}
Then, by subtracting \eqref{eq2} from \eqref{eq3}, we obtain that:
$$
P(n) = \binom{n + d}{d} f(n + d) - \binom{n}{d} f(n) ~~~~~~~~~~ (\forall n \in \N) .
$$
Finally, since the two sides of the last formula are both polynomials in $n$ then its validity for any $n \in \N$ implies its validity as a polynomial identity; that is
$$
P(X) = \binom{X + d}{d} f(X + d) - \binom{X}{d} f(X) ,
$$ 
as required. \\[1mm]
\textbullet{} Conversely, suppose that $P$ has the form \eqref{eqI}. Then, for any $n , q \in \N$, we have
\begin{align*}
\sum_{k = 1}^{q} P(n - k d) & = \underbrace{\sum_{k = 1}^{q} \left\{\binom{n - (k - 1) d}{d} f(n - (k - 1) d) - \binom{n - k d}{d} f(n - k d)\right\}}_{\text{telescopic sum}} \\
& = \binom{n}{d} f(n) - \binom{n - q d}{d} f(n - q d) .
\end{align*}
By taking in particular $q = \lfloor\frac{n}{d}\rfloor$ and putting $r := n - q d$ (so $r$ is the remainder of the euclidean division of $n$ by $d$, and consequently it is the smallest natural number among the integers $n , n - d , n - 2 d , n - 3 d , \dots$), we find that:
$$
P(n - d) + P(n - 2 d) + P(n - 3 d) + \dots = \binom{n}{d} f(n) - \binom{r}{d} f(r) .
$$
But since $r \in \{0 , 1 , \dots , d - 1\}$, we have $\binom{r}{d} = 0$; hence
$$
P(n - d) + P(n - 2 d) + P(n - 3 d) + \dots = \binom{n}{d} f(n) ,
$$
showing that the sum $P(n - d) + P(n - 2 d) + P(n - 3 d) + \dots$ is polynomial in $n$; thus $P \in \E_d$, as required. This completes the proof of the theorem.
\end{proof}

From Theorem \ref{t1}, we derive the following corollary which determines an explicit basis for the $\R$-vector space $\E_d$ ($d \in \N^*$).
\begin{coll}\label{coll1}
Let $d$ be a positive integer. Then the family of real polynomials
$$
\mathscr{B} := {\left(\binom{X + d}{d} (X + d)^k - \binom{X}{d} X^k\right)}_{k \in \N}
$$
constitutes a basis for the $\R$-vector space $\E_d$.
\end{coll}
\begin{proof} ~\\
\textbullet{} First, let us show that $\mathscr{B}$ generates $\E_d$. For any given $P \in \E_d$, we can write (according to Theorem \ref{t1}):
$$
P(X) = \binom{X + d}{d} f(X + d) - \binom{X}{d} f(X) ,
$$
for some $f \in \R[X]$. Then, by expressing $f$ in the canonical basis of $\R[X]$; that is in the form:
$$
f(X) = \sum_{k = 0}^{n} a_k X^k
$$
(with $n \in \N$ and $a_0 , a_1 , \dots , a_n \in \R$), we obtain that:
\begin{align*}
P(X) & = \binom{X + d}{d} \sum_{k = 0}^{n} a_k (X + d)^k - \binom{X}{d} \sum_{k = 0}^{n} a_k X^k \\[1mm]
& = \sum_{k = 0}^{n} a_k \left[\binom{X + d}{d} (X + d)^k - \binom{X}{d} X^k\right] ,
\end{align*}
which is an expression of $P$ as a linear combination (with real coefficients) of elements of $\mathscr{B}$. Therefore, $\mathscr{B}$ generates $\E_d$. \\[1mm]
\textbullet{} Now, let us show that $\mathscr{B}$ is a free family in the $\R$-vector space $\R[X]$. Let $K$ be a finite nonempty subset of $\N$ and ${(\lambda_k)}_{k \in K}$ be a family of real numbers such that:
\begin{equation}\label{eq*}
\sum_{k \in K} \lambda_k \left(\binom{X + d}{d} (X + d)^k - \binom{X}{d} X^k\right) = 0 .
\end{equation}
So, we have to show that $\lambda_k = 0$ for every $k \in K$. By putting $Q(X) := \binom{X}{d} \sum_{k \in K} \lambda_k X^k \in \R[X]$, we have
$$
\eqref{eq*} \Longleftrightarrow Q(X + d) = Q(X) \Longleftrightarrow Q \text{ is $d$-periodic} . 
$$
But if $Q$ is not zero, we have that $\deg Q \geq \deg \binom{X}{d} = d \geq 1$; so $Q$ cannot be $d$-periodic (because a real continuous function on $\R$ which is $d$-periodic is inevitably bounded, while a real polynomial with degree $\geq 1$ is never bounded). Thus $Q$ is zero, implying that $\lambda_k = 0$ for every $k \in K$. Consequently, $\mathscr{B}$ is free. \\
In conclusion, $\mathscr{B}$ is a basis of the $\R$-vector space $\E_d$. The corollary is proved.
\end{proof}

\noindent\textbf{An important example (the case $\boldsymbol{d = 2}$):} \\
The particular case $d = 2$ is the one to which Ibn al-Banna's formula corresponds, and for this reason, it requires more attention. By applying Corollary \ref{coll1} for $d = 2$, we find that the $\R$-vector space $\E_2$ has as a basis the family of polynomials ${(e_k)}_{k \in \N}$, defined by:
$$
e_k(X) := \binom{X + 2}{2} (X + 2)^k - \binom{X}{2} X^k ~~~~~~~~~~ (\forall k \in \N) .
$$
The calculations give
\begin{align*}
e_0(X) & = 2 X + 1 , \\
e_1(X) & = 3 X^2 + 4 X + 2 , \\
e_2(X) & = 4 X^3 + 9 X^2 + 10 X + 4 , \text{ etc.}
\end{align*}
In particular, we observe that the polynomial $X^2$ (involved in Ibn al-Banna's formula) can be written as:
$$
X^2 = \frac{1}{3} \left(e_1 - 2 e_0\right) ,
$$
which is a linear combination of the $e_k$'s, so belongs to $\E_2$. Next, we observe that:
$$
X^3 - \frac{1}{4} = \frac{1}{4} \left(e_2 - 3 e_1 + e_0\right) ,
$$
showing that $\left(X^3 - \frac{1}{4}\right) \in \E_2$. More generally, we will show latter (see Corollary \ref{coll3} and Theorem \ref{t2}) that for all natural number $k$, we have
$$
\big(X^k - E_k(0)\big) \in \E_2 .
$$
In addition, the family of polynomials ${\big(X^k - E_k(0)\big)}_{k \in \N^*}$ constitutes a basis for the $\R$-vector space $\E_2$ (simpler than ${(e_k)}_{k \in \N}$). 

In what follows, we derive from Theorem \ref{t1} a ``simple'' complement subspace (in $\R[X]$) of the $\R$-vector space $\E_d$ ($d \in \N^*$). We have the following corollary:
\begin{coll}\label{coll2}
Let $d$ be a positive integer. Then the $\R$-vector space $\R_{d - 2}[X]$ is a complement subspace {\rm(}in $\R[X]${\rm)} of $\E_d$; that is
$$
\E_d \oplus R_{d - 2}[X] = \R[X] .
$$
In particular, we have
$$
\codim_{\R[X]} \E_d = d - 1 .
$$
\end{coll}
\begin{proof}
We have to show that $\E_d \cap \R_{d - 2}[X] = \{0_{\R[X]}\}$ and that $\E_d + \R_{d - 2}[X] = \R[X]$. \\[1mm]
\textbullet{} Let us first show that $\E_d \cap \R_{d - 2}[X] = \{0_{\R[X]}\}$. So, let $P \in \E_d \cap \R_{d - 2}[X]$ and let us show that $P$ is necessarily the zero polynomial (the second inclusion $\{0_{\R[X]}\} \subset \E_d \cap \R_{d - 2}[X]$ is trivial). From $P \in \E_d$, we deduce (by Theorem \ref{t1}) that there is $f \in \R[X]$ such that:
$$
P(X) = \binom{X + d}{d} f(X + d) - \binom{X}{d} f(X) .
$$ 
If we suppose $P \neq 0_{\R[X]}$, we would have $f \neq 0_{\R[X]}$ and
$$
\deg P = \deg\left(\binom{X}{d} f(X)\right) - 1 = d + \deg f - 1 \geq d - 1 ,
$$
which contradicts the fact that $P \in \R_{d - 2}[X]$. Hence $P = 0_{\R[X]}$, as required. This confirms that $\E_d \cap \R_{d - 2}[X] = \{0_{\R[X]}\}$. \\
\textbullet{} Now, let us show that $\E_d + \R_{d - 2}[X] = \R[X]$. So, let $P \in \R[X]$ and let us show the existence of two polynomials $Q \in \E_d$ and $R \in \R_{d - 2}[X]$ such that $P = Q + R$ (the other inclusion $\E_d + \R_{d - 2}[X] \subset \R[X]$ is trivial). To do so, let us consider $P^* \in \R[X]$ such that:
$$
P^*(X + 1) - P^*(X) = P(d X)
$$
(such $P^*$ exists because $P(d X) \in \R[X]$). Let us then consider the euclidean division (in $\R[X]$) of $P^*$ by the polynomial $\binom{d X}{d}$ (which is of degree $d$):
\begin{equation}\label{eq4}
P^*(X) = \binom{d X}{d} q(X) + r(X) ,
\end{equation}
where $q , r \in \R[X]$ and $\deg r \leq d - 1$. Using \eqref{eq4}, we have that
$$
P^*(X + 1) - P^*(X) = \binom{d X + d}{d} q(X + 1) - \binom{d X}{d} q(X) + r(X + 1) - r(X) ;
$$
that is
$$
P(d X) = \binom{d X + d}{d} q(X + 1) - \binom{d X}{d} q(X) + r(X + 1) - r(X) .
$$
By substituting in this last equality $X$ by $\frac{X}{d}$, we get
$$
P(X) = \binom{X + d}{d} q\left(\frac{X + d}{d}\right) - \binom{X}{d} q\left(\frac{X}{d}\right) + r\left(\frac{X}{d} + 1\right) - r\left(\frac{X}{d}\right) .
$$
So, it suffices to take
\begin{align*}
Q(X) & := \binom{X + d}{d} q\left(\frac{X + d}{d}\right) - \binom{X}{d} q\left(\frac{X}{d}\right) \\[-5mm]
\intertext{and} \\[-1.5cm]
R(X) & := r\left(\frac{X}{d} + 1\right) - r\left(\frac{X}{d}\right)
\end{align*}
to have $P = Q + R$, with $Q \in \E_d$ (according to Theorem \ref{t1}) and $R \in \R_{d - 2}[X]$ (since $\deg R = \deg r - 1 \leq d - 2$). Consequently, we have $\R[X] = \E_d + \R_{d - 2}[X]$, as required. This completes the proof of the corollary.
\end{proof}

From Corollary \ref{coll2}, we derive the following important corollary relating to the particular case $d = 2$.

\begin{coll}\label{coll3}
There is a unique real sequence ${(c_n)}_{n \in \N}$ such that:
$$
\left(X^n - c_n\right) \in \E_2 ~~~~~~~~~~ (\forall n \in \N) .
$$
Besides, we have $c_{2 n} = 0$ for any positive integer $n$.
\end{coll}
\begin{proof}
By Corollary \ref{coll2}, we have
$$
\R[X] = \E_2 \oplus \R_0[X] .
$$
Since $\R_0[X]$ is constituted by constant polynomials, this shows in particular that for any $n \in \N$, the monomial $X^n$ can be written in a unique way as $X^n = P_n(X) + c_n$, with $P_n \in \E_2$ and $c_n \in \R$. In other words, for any $n \in \N$, there is a unique $c_n \in \R$ such that $(X^n - c_n) \in \E_2$. This confirms the first part of the corollary. Next, for any $n \in \N^*$, the immediate formula:
$$
N^{2 n} + (N - 2)^{2 n} + (N - 4)^{2 n} + \dots = \frac{1}{2} \sum_{k = 0}^{N} (N - 2 k)^{2 n} ~~~~~~~~~~ (\forall N \in \N)
$$
insures that $X^{2 n} \in \E_2$; hence $c_{2 n} = 0$. This confirms the second part of the corollary and achieves the proof.
\end{proof}

For the sequel, we let ${(c_n)}_{n \in \N}$ denote the real sequence established by Corollary \ref{coll3}. The following theorem provides explicit expressions of the $c_n$'s ($n \in \N$), using the Euler polynomials and the Genocchi numbers.

\begin{thm}\label{t2}
For every natural number $n$, we have
$$
c_n = E_n(0) = \frac{G_{n + 1}}{n + 1} .
$$
\end{thm}
\begin{proof}
Let $n$ be a fixed natural number. We exactly follow the second part of the proof of Corollary \ref{coll2} (i.e., the proof of the fact that $\R[X] = \E_d + \R_{d - 2}[X]$) by taking $d = 2$ and $P(X) = X^n$. We thus need a polynomial $P^* \in \R[X]$ such that:
$$
P^*(X + 1) - P^*(X) = P(2 X) = (2 X)^n = 2^n X^n .
$$
According to Formula \eqref{eqJ}, we can take
$$
P^*(X) = \frac{2^n}{n + 1} B_{n + 1}(X) .
$$
Then, consider the euclidean division (in $\R[X]$) of this chosen polynomial $P^*(X)$ by the polynomial $\binom{2 X}{2} = X (2 X - 1)$:
$$
P^*(X) = X (2 X - 1) q(X) + r(X) ,
$$
where $q , r \in \R[X]$ and $\deg r \leq 1$. Expressing $r(X)$ as $r(X) = a X + b$ ($a , b \in \R$), we get the polynomial identity:
\begin{equation}\label{eq5}
\frac{2^n}{n + 1} B_{n + 1}(X) = X (2 X - 1) q(X) + a X + b .
\end{equation}
According to the second part of the proof of Corollary \ref{coll2}, we have
$$
c_n = R(X) := r\left(\frac{X}{2} + 1\right) - r\left(\frac{X}{2}\right) = a \left(\frac{X}{2} + 1\right) + b - \left(a \frac{X}{2} + b\right) = a .
$$
So, we have to show that $a = E_n(0) = \frac{G_{n + 1}}{n + 1}$. By taking in \eqref{eq5} successively $X = 0$ and $X = \frac{1}{2}$, we obtain that:
\begin{align*}
b & = \frac{2^n}{n + 1} B_{n + 1}(0) = \frac{2^n}{n + 1} B_{n + 1} , \\[-5mm]
\intertext{and} \\[-13mm]
\frac{a}{2} + b & = \frac{2^n}{n + 1} B_{n + 1}\left(\frac{1}{2}\right) \\[1mm]
& = \frac{2^n}{n + 1} \left(\frac{1}{2^n} - 1\right) B_{n + 1} ~~~~~~~~~~ (\text{according to \eqref{eqJJ}}) \\[1mm]
& = \frac{1 - 2^n}{n + 1} B_{n + 1} .
\end{align*} 
Thus
\begin{align*}
a & = 2 \left[\left(\frac{a}{2} + b\right) - b\right] \\[1mm]
& = 2 \left[\frac{1 - 2^n}{n + 1} B_{n + 1} - \frac{2^n}{n + 1} B_{n + 1}\right] \\[1mm]
& = \frac{2 \left(1 - 2^{n + 1}\right)}{n + 1} B_{n + 1} \\[1mm]
& = E_n(0) = \frac{G_{n + 1}}{n + 1} ~~~~~~~~~~ (\text{according to \eqref{eqJJJ}}) ,
\end{align*}
as required. This completes the proof.
\end{proof}

\begin{rmks}~\\[-1cm]
\begin{enumerate}
\item Since, as known, $G_n = 0$ for any odd integer $n \geq 3$, we deduce from Theorem \ref{t2} the fact that $c_n = 0$ for any even integer $n \geq 2$, which is already announced by Corollary \ref{coll3} (but proved in a different way).
\item Using the formula $c_n = \frac{G_{n + 1}}{n + 1}$ (of Theorem \ref{t2}) together with the fact that each Genocchi number $G_n$ ($n \in \N^*$) is a multiple of the odd part of $n$ (see e.g., \cite[Corollary 2.2]{far2}), we easily show that the $c_n$'s are all rational numbers with denominators powers of $2$. Precisely, for any $n \in \N$, we can write $c_n = \frac{a}{2^e}$, with $a \in \Z$, $e \in \N$ and $e \leq \vartheta_2(n + 1)$. 
\end{enumerate}
\end{rmks}

The following table provides the values of the first numbers $c_n$ ($n = 0 , 1 , \dots, 15$):

\bigskip

\begin{center}
\begin{tabular}{c|c|c|c|c|c|c|c|c|c|c|c|c|c|c|c|c|}
$n$ & $0$ & $1$ & $2$ & $3$ & $4$ & $5$ & $6$ & $7$ & $8$ & $9$ & $10$ & $11$ & $12$ & $13$ & $14$ & $15$ \\
\hline
$c_n$ & $1$ & $- \frac{1}{2}$ & $0$ & $\frac{1}{4}$ & $0$ & $- \frac{1}{2}$ & $0$ & $\frac{17}{8}$ & $0$ & $- \frac{31}{2}$ & $0$ & $\frac{691}{4}$ & $0$ & $- \frac{5461}{2}$ & $0$ & $\frac{929569}{16}$
\end{tabular}
\end{center}

\bigskip

Now, we will derive from our main theorem \ref{t1} that the Ibn al-Banna method (exposed in §\ref{sec1}) for expressing in a closed form a sum of the type $P(n) + P(n - d) + P(n - 2 d) + \dots$, when $d \geq 2$ is an integer and $P \in \E_d$, is in fact general. Precisely, we will show that in a such context, it is always possible to express $P$ in the form $P(X) = \alpha(X) + \alpha(X - 1) + \dots + \alpha(X - d + 1)$, where $\alpha \in \R[X]$ and $\alpha(-1) = \alpha(-2) = \dots = \alpha(- d + 1) = 0$.

\begin{thm}[Generalizability of the Ibn al-Banna method]
Let $d \geq 2$ be an integer. Then a real polynomial $P$ belongs to $\E_d$ if and only if there exists $\alpha \in \R[X]$, satisfying $\alpha(-1) = \alpha(-2) = \dots = \alpha(- d + 1) = 0$, such that:
$$
P(X) = \alpha(X) + \alpha(X - 1) + \dots + \alpha(X - d + 1) .
$$
\end{thm}

\begin{proof}
Let $P \in \R[X]$ be fixed. \\[1mm]
\textbullet{} Suppose that $P \in \E_d$ and let us show that $P$ is of the form required by the theorem. According to Theorem \ref{t1}, there exists $f \in \R[X]$ such that:
$$
P(X) = \binom{X + d}{d} f(X + d) - \binom{X}{d} f(X) .
$$
Now, consider $\alpha \in \R[X]$ defined by:
$$
\alpha(X) := \binom{X + d}{d} f(X + d) - \binom{X + d - 1}{d} f(X + d - 1) .
$$
So, we have clearly
$$
\alpha(-1) = \alpha(-2) = \dots = \alpha(- d + 1) = 0 .
$$
Next, we have
\begin{align*}
\alpha(X) + \alpha(X - 1) + \dots + \alpha(X - d + 1) & = \sum_{k = 0}^{d - 1} \alpha(X - k) \\
& \hspace*{-4cm} = \underbrace{\sum_{k = 0}^{d - 1} \left\{\binom{X - k + d}{d} f(X - k + d) - \binom{X - k + d - 1}{d} f(X - k + d - 1)\right\}}_{\text{telescopic sum}} \\
& \hspace*{-4cm} = \binom{X + d}{d} f(X + d) - \binom{X}{d} f(X) \\
& \hspace*{-4cm} = P(X) .
\end{align*}
Therefore $P$ has the form required by the theorem. \\[1mm]
\textbullet{} Conversely, suppose that there is $\alpha \in \R[X]$, satisfying $\alpha(-1) = \alpha(-2) = \dots = \alpha(- d + 1) = 0$, such that:
$$
P(X) = \alpha(X) + \alpha(X - 1) + \dots + \alpha(X - d + 1)
$$
and let us show that $P \in \E_d$. For a given $n \in \N$, denoting by $r_n$ the remainder of the euclidean division of $n$ by $d$ (so $r_n \in \{0 , 1 , \dots , d - 1\}$), we have
\begin{align*}
P(n) + P(n - d) + P(n - 2 d) + \dots & = P(n) + P(n - d) + P(n - 2 d) + \dots + P(r_n) \\
& = \phantom{+} \alpha(n) + \alpha(n - 1) + \dots + \alpha(n - d + 1) \\[-3mm]
& \phantom{=} ~+ \\[-3mm]
& \phantom{= +} ~\alpha(n - d) + \alpha(n - d - 1) + \dots + \alpha(n - 2 d + 1) \\[-3mm]
& \phantom{=} ~+ \\[-3mm]
& \phantom{= +} ~\alpha(n - 2 d) + \alpha(n - 2 d - 1) + \dots + \alpha(n - 3 d + 1) \\[-3mm]
& \phantom{=} ~+ \\[-3mm]
& \phantom{=} ~~\vdots \\[-3mm]
& \phantom{=} ~+ \\[-3mm]
& \phantom{= +} ~\alpha(r_n) + \alpha(r_n - 1) + \dots + \alpha(r_n - d + 1) \\[1mm]
& = \alpha(n) + \alpha(n - 1) + \dots + \alpha(r_n - d + 1) .
\end{align*}
But since $- d + 1 \leq r_n - d + 1 \leq 0$ and $\alpha(k) = 0$ for $k = - 1 , - 2 , \dots , - d + 1$, it follows that:
$$
P(n) + P(n - d) + P(n - 2 d) + \dots = \alpha(n) + \alpha(n - 1) + \dots + \alpha(0) ,
$$
which is polynomial in $n$ (since $\alpha \in \R[X]$). This concludes that $P \in \E_d$, as required. The proof of the theorem is complete.
\end{proof}

\subsection*{Study of the spaces $\E_d$ through formal power series and generalization}
We briefly present in what follows a new approach (using convolution products and formal power series) to studying the $\R$-vector space $\E_d$ ($d \in \N^*$) and expressing in a closed form the sums of the type:
$$
P(n) + P(n - d) + P(n - 2 d) + \dots ~~~~~~~~~~ (P \in \R[X] , n \in \N) .
$$
Then, we lean on the obtained result in order to generalize the spaces $\E_d$. The results concerning the generalized spaces in question, which are analog to the above results on the $\E_d$'s, are given without proofs.

By definition, \textit{the ordinary convolution product} of two real sequences $\mathbf{a} = {(a_n)}_{n \in \N}$ and $\mathbf{b} = {(b_n)}_{n \in \N}$ is the real sequence ${(\mathbf{a} * \mathbf{b})}_{n \in \N}$, defined by:
$$
{(\mathbf{a} * \mathbf{b})}_n := \sum_{k = 0}^{n} a_k b_{n - k} ~~~~~~~~~~ (\forall n \in \N) .
$$
It is known that the convolution product constitutes a commutative and associative law of composition on the set of the real sequences. In addition, the Dirac delta sequence $\boldsymbol{\delta} = {(\delta_n)}_{n \in \N}$, defined by: $\delta_0 = 1$ and $\delta_n = 0$ ($\forall n \geq 1$) is the neutral element of $*$. Finally, every real sequence ${(a_n)}_{n \in \N}$ with $a_0 \neq 0$ has an inverse element for $*$. Further, \textit{the ordinary generating function} of a given real sequence $\mathbf{u} = {(u_n)}_{n \in \N}$ is the formal power series defined by:
$$
\sigma(\mathbf{u}) := \sum_{n = 0}^{+ \infty} u_n X^n .
$$
When $\mathbf{a} = {(a_n)}_{n \in \N}$ and $\mathbf{b} = {(b_n)}_{n \in \N}$ are two real sequences, we have the fundamental formula:
$$
\sigma(\mathbf{a}) \cdot \sigma(\mathbf{b}) = \sigma(\mathbf{a} * \mathbf{b}) .
$$
Next, for a positive integer $d$, let $\boldsymbol{\delta}^{(d)} = {(\delta_n^{(d)})}_{n \in \N}$ denote the real sequence defined by:
$$
\delta_n^{(d)} := \begin{cases}
1 & \text{if $n$ is a multiple of $d$} \\
0 & \text{otherwise}
\end{cases} ~~~~~~~~~~ (\forall n \in \N) .
$$
(Notice that $\boldsymbol{\delta}^{(1)} \equiv 1$ and $\boldsymbol{\delta}^{(+ \infty)} = \boldsymbol{\delta}$). So, the ordinary generating function of $\boldsymbol{\delta}^{(d)}$ ($d \in \N^*$) is given by:
\begin{align*}
\sigma(\boldsymbol{\delta}^{(d)}) & := \sum_{n = 0}^{+ \infty} \delta_n^{(d)} X^n \\
& = \sum_{\begin{subarray}{c}
n \in \N \\
n \equiv \restmod{0}{d}
\end{subarray}} X^n \\
& = \sum_{k = 0}^{+ \infty} X^{d k} ;
\end{align*}
that is
\begin{equation}\label{eqalpha}
\sigma(\boldsymbol{\delta}^{(d)}) = \frac{1}{1 - X^d} .
\end{equation}

For simplicity, for $P \in \R[X]$ and $d \in \N^*$, we let $S_{P , d}$ denote the real sequence of general term $S_{P , d}(n)$ ($n \in \N$) defined by:
$$
S_{P , d}(n) := P(n) + P(n - d) + P(n - 2 d) + \dots
$$
(recall that the sum on the right stops at $P(r)$, where $r$ is the remainder of the euclidean division of $n$ by $d$). So, given $d$ a positive integer and $P$ a real polynomial, we have for any natural number $n$:
\begin{align*}
S_{P , d}(n) & := P(n) + P(n - d) + P(n - 2 d) + \dots \\
& = \sum_{k = 0}^{n} \delta_k^{(d)} P(n - k) \\
& = \left(\boldsymbol{\delta}^{(d)} * P\right)(n) ;
\end{align*}
hence the fundamental formula of this approach:
\begin{equation}\label{eqbeta}
S_{P , d} = \boldsymbol{\delta}^{(d)} * P .
\end{equation}

We also need for this approach the following well-known theorem which characterizes the ordinary generating functions of polynomial sequences.

\begin{thmn}[{\cite[Corollary 4.3.1]{sta}}]\label{t-T}
Let ${(a_n)}_{n \in \N}$ be a non-zero real sequence. Then ${(a_n)}_n$ is polynomial in $n$ if and only if its ordinary generating function is a rational function of the form $\frac{U(X)}{(1 - X)^{\alpha}}$, with $\alpha \in \N^*$, $U \in \R[X]$, $\deg U < \alpha$ and $U(1) \neq 0$. In addition, in the situation where $\sigma(a_n) = \frac{U(X)}{(1 - X)^{\alpha}}$ {\rm(}with the above conditions on $U$ and $\alpha${\rm)}, the degree of $a_n$ {\rm(}as a polynomial in $n${\rm)} is equal to $(\alpha - 1)$.
\end{thmn}

According to Theorem \ref{t-T}, it is associated to every polynomial $P \in \R[X]$ a unique polynomial $A_P \in \R[X]$ satisfying the identity of formal power series:
\begin{equation}\label{eqgamma}
\sum_{n = 0}^{+ \infty} P(n) X^n = \frac{A_P(X)}{(1 - X)^{\deg P + 1}} .
\end{equation}
If, in addition, $P \neq 0_{\R[X]}$ then we have $A_P \neq 0_{\R[X]}$, $A_P(1) \neq 0$ and $\deg A_P \leq \deg P$. The correspondence $P \mapsto A_P$ provides a simple characterization of the property of belonging to a certain space $\E_d$ for a given real polynomial. This characterization inspires us a generalization of the $\R$-vector spaces $\E_d$. First, we have the following proposition:
\begin{prop}\label{p1}
Let $d$ be a positive integer and $P$ be a real polynomial. Then $P$ belongs to $\E_d$ if and only if $A_P$ is a multiple {\rm(}in $\R[X]${\rm)} of the polynomial $(1 + X + X^2 + \dots + X^{d - 1})$.
\end{prop}
\begin{proof}
The result of the proposition is obvious for $P = 0_{\R[X]}$. Suppose for the sequel of this proof that $P \neq 0_{\R[X]}$ and set $s := \deg P \in \N$. By using successively Formulas \eqref{eqbeta}, \eqref{eqalpha} and \eqref{eqgamma}, we get
\begin{align}
\sigma(S_{P , d}) & = \sigma(\boldsymbol{\delta}^{(d)} * P) \notag \\
& = \sigma(\boldsymbol{\delta}^{(d)}) \cdot \sigma(P) \notag \\
& = \frac{1}{1 - X^d} \sum_{n = 0}^{+ \infty} P(n) X^n \notag \\
& = \frac{1}{(1 - X) (1 + X + X^2 + \dots + X^{d - 1})} \cdot \frac{A_P(X)}{(1 - X)^{s + 1}} \notag \\
& = \frac{A_P(X)}{(1 + X + X^2 + \dots + X^{d - 1}) (1 - X)^{s + 2}} . \label{eqdelta}
\end{align}
\textbullet{} Suppose that $A_P$ is a multiple (in $\R[X]$) of the polynomial $(1 + X + X^2 + \dots + X^{d - 1})$ and let $U(X) := \frac{A_P(X)}{1 + X + X^2 + \dots + X^{d - 1}} \in \R[X]$. Then $\deg U = \deg A_P - (d - 1) \leq \deg A_P \leq \deg P = s$ and $U(1) = \frac{A_P(1)}{d} \neq 0$. Next, Formula \eqref{eqdelta} becomes: $\sigma(S_{P , d}) = \frac{U(X)}{(1 - X)^{s + 2}}$, which implies (according to Theorem \ref{t-T}) that $S_{P , d}(n)$ is polynomial in $n$; that is $P \in \E_d$. \\[1mm]
\textbullet{} Conversely, suppose that $P \in \E_d$; that is $S_{P , d}(n)$ is polynomial in $n$. By Theorem \ref{t-T}, the power series $\sigma(S_{P , d})$ has the form:
$$
\sigma\left(S_{P , d}\right) = \frac{U(X)}{(1 - X)^{\alpha}} ,
$$
where $\alpha \in \N^*$, $U \in \R[X]$, $\deg U < \alpha$ and $U(1) \neq 0$. It follows (according to \eqref{eqdelta}) that:
$$
\frac{U(X)}{(1 - X)^{\alpha}} = \frac{A_P(X)}{\left(1 + X + X^2 + \dots + X^{d - 1}\right) \left(1 - X\right)^{s + 2}} .
$$
Thus
$$
\left(1 + X + X^2 + \dots + X^{d - 1}\right) \left(1 - X\right)^{s + 2} U(X) = \left(1 - X\right)^{\alpha} A_P(X) .
$$
This shows in particular that $(1 + X + X^2 + \dots + X^{d - 1}) \mid (1 - X)^{\alpha} A_P(X)$. But since $(1 + X + X^2 + \dots + X^{d - 1})$ is coprime with $(1 - X)^{\alpha}$ (because $(1 + X + X^2 + \dots + X^{d - 1})$ does not vanish at $X = 1$), it follows (according to Gauss's lemma) that $(1 + X + X^2 + \dots + X^{d - 1}) \mid A_P(X)$, as required. This completes the proof of the proposition.
\end{proof}

By carefully observing the proof of Proposition \ref{p1}, we see that the polynomial $(1 + X + X^2 + \dots + X^{d - 1})$ which appears there has no particularity apart from the fact that it does not vanish at $X = 1$. This inspires us a generalization of the spaces $\E_d$ in the following way. 

\begin{defi}[Generalization of the spaces $\E_d$]
Let $D \in \R[X]$ such that $D(1) \neq 0$. We define $\E_D$ as the set of all polynomials $P \in \R[X]$ for which the polynomial $A_P$ is a multiple of $D$ (in $\R[X]$).
\end{defi}

Notice that for all positive integer $d$, we have (according to Proposition \ref{p1})
$$
\E_d = \E_{(1 + X + X^2 + \dots + X^{d - 1})} .
$$

Further, we show quite easily that for all $D \in \R[X]$, with $D(1) \neq 1$, the set $\E_D$ constitutes a $\R$-linear subspace of $\R[X]$. A basis of $\E_D$ (as a $\R$-vector space) is specified by the following theorem, the proof of which (left to the reader's imagination) mainly uses Theorem \ref{t-T}.

\begin{thm}\label{t3}
Let $d$ be a positive integer and $D$ be a real polynomial of degree $(d - 1)$ such that $D(1) \neq 0$. Let $h \in \R[X]$ defined by the formal power series identity:
$$
\frac{D(X)}{(1 - X)^{\alpha}} = \sum_{n = 0}^{+ \infty} h(n) X^n .
$$
Then, the family of polynomials $(h , h * 1 , h * X , h * X^2 , \dots)$ constitutes a basis for the $\R$-vector space $\E_D$. \hfill $\square$ 
\end{thm}

\begin{rmk}
We can derive our main theorem \ref{t1} on the spaces $\E_d$ from the general theorem \ref{t3} as follows.
\begin{itemize}
\item First, show that we have for all positive integer $d$:
$$
\frac{1 + X + X^2 + \dots + X^{d - 1}}{(1 - X)^d} = \sum_{n = 0}^{+ \infty} \left(\binom{n + d}{d} - \binom{n}{d}\right) X^n .
$$
\item (Less easy!). Show the following statement:
\begin{quote}
\textit{%
For all positive integer $d$ and all polynomial $f \in \R[X]$, there exists a unique polynomial $g \in \R[X]$, which vanishes at $(-1)$, such that:
$$
\left(\binom{n + d}{d} - \binom{n}{d}\right) * f(n) = \binom{n + d}{d} g(n + d) - \binom{n}{d} g(n) .
$$
In addition, the correspondence $f \mapsto g$ from $\R[X]$ to $\{g \in \R[X] : g(-1) = 0\}$ is bijective.
}
\end{quote}
\end{itemize}
\end{rmk}

We finally derive from Theorem \ref{t3} the below corollary, generalizing Corollary \ref{coll2}.

\begin{coll}
In the context of Theorem \ref{t3}, we have
$$
\E_D \oplus \R_{d - 2}[X] = \R[X] .
$$
In particular, we have
\begin{equation}
\codim_{\R[X]} \E_D = d - 1 . \tag*{$\square$}
\end{equation}
\end{coll}

\end{document}